\documentclass[11 pt]{amsart}

\usepackage{graphicx} 
\usepackage{amsfonts}
\usepackage{tikz-cd}
\usepackage{tikz}
\usetikzlibrary{patterns, decorations.pathreplacing, arrows.meta}
\usepackage{tikz-3dplot}
\usepackage{tcolorbox}
\usepackage{anyfontsize}
\usepackage{tabularx}
\usepackage{scalerel}
\usepackage{bm}
\usepackage{tikz}
\usetikzlibrary{automata, arrows.meta, positioning}
\usetikzlibrary{shapes,shapes.geometric,arrows,fit,calc,positioning,automata,}
\usepackage{amssymb}

\usepackage{amsmath}
\usepackage{amsthm}
\usepackage{aligned-overset}
\usepackage{tabto}
\usepackage{mathtools}
\usepackage{marginnote}
\usepackage{MnSymbol}
\usepackage{enumerate}

\usepackage{cite} 
\usepackage{float} 
\newtheorem{theorem}{Theorem}[section]
\newtheorem{proposition}[theorem]{Proposition}
\newtheorem{lemma}[theorem]{Lemma}
\newtheorem{definition}[theorem]{Definition}

\newtheorem{theoremAlph}{Theorem}

\newtheorem{question}{Question}

\newtheorem*{theorem*}{Theorem}
\newtheorem*{lemma*}{Lemma}
\newtheorem*{proposition*}{Proposition}
\newtheorem*{corollary*}{Corollary}
\newtheorem*{definition*}{Definition}

\title{Hypercube minor-universality}
\author{\textsc{Itai Benjamini}, \,  \textsc{Or Kalifa}, \, \and \, Elad Tzalik }
\date{January 2025}
\address{The Weizmann Institute of Science}
\email{itai.benjamini@gmail.com}
\email{or1234e@gmail.com}
\email{tzalikemail@gmail.com}

\usepackage[margin=1.5cm]{geometry}
\usepackage{setspace}
\usepackage{times}

\usepackage[hyperfootnotes=true]{hyperref}
\hypersetup{
    colorlinks=False, 
    linktoc=all,     
}

\begin{document}

\newcommand{\R}{\mathbb{R}}
\newcommand{\N}{\mathbb{N}}
\newcommand{\Z}{\mathbb{Z}}
\newcommand{\G}{\Gamma}
\newcommand{\ra}{\rightarrow}
\newcommand{\g}{\gamma}
\newcommand{\s}{\sigma}
\newcommand{\ps}{\psi}
\newcommand{\vp}{\varphi}
\newcommand{\GH}{G_\mathbb{H}}
\newcommand{\fullGraph}{P_3}
\newcommand{\GraphWithTwoAdditionals}{P_2}
\newcommand{\GraphWithOneAdditional}{P_1}
\newcommand{\GraphNaked}{P_0}

\maketitle
\begin{abstract}
A graph \( G \) is  \( m \)-minor-universal if every graph with at most \( m \) edges (and no isolated vertices) is a minor of \( G \). We prove that the \( d \)-dimensional hypercube, \( Q_d \), is \( \Omega\left(\frac{2^d}{d}\right) \)-minor-universal, and that there exists an absolute constant \( C >0 \) such that \( Q_d \) is not \( \frac{C2^d}{\sqrt{d}} \)-minor-universal. Similar results are obtained in a more generalized setting, where we bound the size of minors in a product of finite connected graphs. A key component of our proof is the following claim regarding the decomposition of a permutation of a box into simpler, one-dimensional permutations: Let \( n_1, \dots, n_d \) be positive integers, and define \( X := [n_1] \times \dots \times [n_d] \). We prove that every permutation \( \sigma: X \to X \) can be expressed as \( \sigma = \sigma_1 \circ \dots \circ \sigma_{2d-1} \), where each \( \sigma_i \) is a \emph{one-dimensional} permutation, meaning it fixes all coordinates except possibly one.
We discuss future directions and pose open problems. 
\end{abstract}
\tableofcontents{}

\section{Introduction}
\label{sec: Introduction}
By \emph{contraction of an edge} $uv$ in a graph $G$, we mean identification of $u$ and $v$, i.e., replacement of $u$ and $v$ by a new vertex $w$ adjacent to all of the neighbors of $u$ and $v$. A graph $H$ is a \textbf{minor} of a graph $G$ if by repeatedly deleting vertices and edges from $G$, as well as contracting edges, one can obtain a graph isomorphic to $H$.
Observe that the minor relation is transitive, that is, if $F \leq H$ and $H\leq G$ then $F \leq G$.

\begin{definition}
A graph $G$ is \textbf{$\bm m$-minor-universal} if $H\leq G$ for every graph $H$ with at most $m$ edges and with no isolated vertices. The \textbf{minor-universality} of $G$, $m(G)$, is the maximal $m\in \N\cup\{0\}$ such that $G$ is $m$-minor-universal.
\end{definition}

We study the existence and non-existence of minors using \textit{combinatorial embeddings}. Determining which graphs can be embedded in graph \( G \) via combinatorial embeddings is nearly equivalent to identifying the minors of \( G \). Section \ref{sec: Combinatorial embedding} formally defines these embeddings and clarifies their precise relationship to minors.

Combinatorial embeddings are a combinatorial adaptation of \emph{thick embeddings}. The study of thick embeddings traces back to the 1960s to the seminal work of Kolmogorov and Barzdin \cite{KB}. As part of their efforts to model physical networks, such as neural networks in the brain or logical networks in computers, they asked how tightly can such networks be realized in 3 dimensional space. They showed that any graph with \( m \) edges (and no isolated vertices) can be embedded in a Euclidean ball of volume \( O(m^{3/2}) \) using a 1-thick embedding. Moreover, to prove this is the best possible they essentially gave the first construction of expander graphs. Even though they worked with the continuous notion of thick embeddings, one can transfer their result to the discrete settings, and get that the 3-dimensional ball of radius $cm$ in $\mathbb{Z}^3$ is $m^2$-minor-universal for large enough $c>0$. 

Since the seminal work of Kolmogorov and Barzdin, thick embeddings were further studied  in  \cite{GG}, \cite{BH}, \cite{Ka} in the context of thick embeddings of graphs and higher dimensional simplicial complexes into various spaces, not necessarily euclidean. 

We highlight that from a combinatorial perspective, embedding minors into graphs, and in particular establishing which graphs are $m$-minor-universal was studied in computer science and combinatorics. In \cite{KR} Kleinberg and Rubinfeld \ proved that an \( \alpha \)-expander with \( n \) vertices and maximum degree \( \Delta \) is \( \Omega\left(\frac{n}{(\log n)^\kappa}\right) \)-minor-universal for some \( \kappa = \kappa(\alpha, \Delta) > 1 \). Krivelevich and Nenadov \cite{Kr} showed that if \( G \) is an \( \alpha \)-expander, then \( m(G) = \Omega\left(\frac{n}{\log n}\right) \), and if \( G \) also has logarithmic girth, then \( m(G) = \Theta\left(\frac{n}{\log n}\right) \). 

There is also a significant body of work on \emph{clique} minors in general graphs and in random graphs. This line of work is motivated by Hadwiger's conjecture, e.g., \cite{BCE, FKO}. Clique minors in graphs without sparse cuts and in expander graphs have been studied in \cite{KN} and \cite{KS}, respectively. Results in \cite{EKK} discuss clique minors in the hypercube itself, and \cite{CS} examine clique minors in product graphs.

\smallskip

 The $d$-dimensional hypercube graph \( Q_d \) is the graph with $d$-bits binary strings (the set $\{0,1\}^d$) as vertices and two binary strings are neighbors if they are of hamming distance $1$. Our main theorem is:

\begin{theoremAlph}[hypercube minor-universality]
 \label{thm: minors in hypercube} 
$m(Q_d)=\Omega\left(\frac{2^d}{d}\right)$ and $m(Q_d)=O\left(\frac{2^d}{\sqrt d}\right)$.
\end{theoremAlph}

In other words, there are absolute constants $C,C'>0$ such that every graph with \( \leq C\frac{2^d}{d} \) edges (and no isolated vertices) is a minor of \( Q_d \), while there exists a graph with \( \leq C'\frac{2^d}{\sqrt d}\) edges (and no isolated vertices) that is not a minor of \( Q_d \).

\smallskip

We highlight that this improves on the result of Kleinberg and Rubinfeld \cite{KR} which guarantees a lower bound \( \Omega\left(\frac{2^d}{d^\kappa}\right) \) for some \( \kappa > 1 \). We also remark that the result of Krivelevich and Nenadov \cite{Kr} gives a lower bound of \( m(Q_d) = \Omega\left(\frac{2^d}{d^2}\right) \), hence our lowerbound is a $d$ factor sharper. This follows from the fact that \( Q_d \) is a \( \Theta(1/\sqrt{d}) \)-expander and the constant in their lower bound depends quadratically on the expansion \( \alpha \).

\begin{question}
The upper and lower bounds in theorem \ref{thm: minors in hypercube} do not match; there is a factor of \( \sqrt{d} \) between them. Is either of these bounds tight, and if so, which one?
\end{question}

\begin{question}
    Given $m\in \N$, what are the graphs with minimal number of edges that are $m$-minor-universal?
\end{question}

\begin{question}
What is the minor-universality of these two variants of the hypercube: the cube-connected cycles\footnote{\url{https://en.wikipedia.org/wiki/Cube-connected_cycles}.} and the randomly twisted hypercube \cite{BDGZ}? Additionally, what is the minor-university of the Cayley graph of $S_n$ with respect to all transpositions?
\end{question}

We tried to slightly generalize the hypercube theorem and found an extension: The cube \( Q_d \) can be expressed as a product of graphs, \( Q_d = Q_1\Box \ldots \Box Q_1 \).\footnote{This is the Cartesian product (also known as the box product), which is properly defined in Section \ref{sec: Lower bound - minors in the hypercube}.}  This naturally leads to the question of whether a similar result holds for a product of arbitrary sequence of graphs. As demonstrated by our next theorem, such a generalization is indeed achievable.

\begin{theoremAlph}[product graph minor-universality]\label{thm: product graph minor-universality}  
Let \( k, n \in \mathbb{N} \), and let \( G_1, \dots, G_n \) be a sequence of graphs, each one with \( \leq k \) vertices.  
\begin{enumerate}  
    \item If the graphs \( G_i \) are connected, then for any pair of connected graphs \( H_4 \) (with 4 vertices) and \( H_k \) (with \( k \) vertices), the graph  
    \[  
    H_4 \Box C_{6n-2} \Box H_k \Box G_1 \Box \cdots \Box G_n  
    \]  
    is \( \frac{1}{16}|V(G_1)| \cdots |V(G_n)| \)-minor-universal.  
    \item There exists a constant \( C = C(k) \) such that the graph \( G_1 \Box \cdots \Box G_n \) is not \( \frac{C k^n}{\sqrt{n}} \)-minor-universal.  
\end{enumerate}  
\end{theoremAlph}

This theorem generalizes theorem \ref{thm: minors in hypercube}. The first part extends the lower bound, the result that \( Q_d \) is \( \Omega\left(\frac{2^d}{d}\right) \)-minor-universal. Although this connection may not be immediately obvious, and the inclusion of the unusual graphs \( H_4, H_k \) and the cycle \( C_{6n-2} \) might seem somewhat unexpected, all of this is explained in detail in Section \ref{sec: Lower bound - minors in the hypercube}. The second part of the theorem generalizes the upper bound, which is more straightforward to recognize.

Note that if we treat \( n \) as a constant (denoted \( d-1 \)) and \( k \) as a variable (denoted \( n \)), we observe that for \( H_n = G_1 = \cdots = G_{d-1} := P_n \) (the path with \( n \) vertices), the theorem implies that \( F \Box B_n \) is \(\Omega(n^{d-1})\)-minor-universal where \( F \) is a fixed finite graph, and \( B_n \) represents the ball of radius \( n \) in \(\mathbb{Z}^d\). This result is similar to the result stating that for every bounded-degree finite graph \(\Gamma\), there exists a \(1\)-thick embedding into \(\mathbb{R}^d\) (for \(d \geq 3\)) with volume \(O(|V\G|^{\frac{d}{d-1}})\) \cite{KB}, \cite{Gu}, \cite{BH}.\footnote{One might wonder if \cite[Theorem 1.3]{BH} implies theorem \ref{thm: minors in hypercube}. While \cite[Theorem 1.3]{BH} demonstrates the embedding of graphs into small balls in \(\mathbb{R}^d\), it does not guarantee embeddings within the hypercube. Moreover, even if \cite[Theorem 1.3]{BH} could be applied to the hypercube, its lower bound on the minor-universality is only $o\left(\frac{2^d}{d}\right)$.}

This raises the question of whether taking the product of a graph with a fixed finite graph significantly impacts its minor-universality. Consider the following question:

\begin{question}
Give sufficient conditions for a sequence of finite graphs \( F , G_1, G_2, \ldots \), to satisfy \(\forall n: m(F \Box G_n) \leq C m(G_n) \) for some $C<\infty$.
\end{question}

For example, this is false when \( G_n \) is a \( n \times n \) planar grid and \( F \) is a single edge. In this case, \( m(G_n) \) is bounded, while \( m(F \Box G_n) \) grows unbounded.

We also believe the following question is an attractive generalization of our main result:

\begin{question}
Let $G$ be a connected transitive graph of diameter $D$, is $Q_1\Box G$ necessarily $\Omega(\frac{|V(G)|}{D})$-minor-universal? Which connected transitive graphs of diameter $D$ are $\Omega(\frac{|V(G)|}{D})$-minor-universal?
\end{question}

\smallskip

A key part of our proof is the following proposition:

\begin{proposition}[box permutation decomposition]
\label{thm: box permutation decomposition}
Let \( d \in \mathbb{N} \) and \( n_1, \dots, n_d \in \mathbb{N} \). Define \( X := [n_1] \times \dots \times [n_d] \). Every permutation \( \sigma: X \to X \) can be expressed as
\[
\sigma = \sigma_1 \circ \dots \circ \sigma_{2d-1},
\]
where \( \sigma_1, \dots, \sigma_{2d-1} \) are one-dimensional permutations. Moreover, if \( n_1, \dots, n_d \geq 2 \), there exists a permutation \( \sigma: X \to X \) such that any decomposition of \( \sigma \) into one-dimensional permutations requires at least \( 2d - 1 \) permutations.
\end{proposition}
A permutation \( \sigma: X \to X \) is called \textbf{one-dimensional} if it affects only one coordinate. Formally, $\s$ is one-dimensional if there exists $i$ such that for every $x\in X$ the points $x$ and $\s(x)$ agree on all coordinate except maybe the $i$-th coordinate.

\subsection{Notations}
\begin{enumerate}[$\bullet$]
    \item We define $[n] := \{1, \dots, n\}$.
    \item $\log n := \log_2 n$.
    \item All the graphs will be finite and simple (without multi-edges and self loops).
    \item Matching graph is graph with maximal degree at most 1.
    \item If $l$ is some finite sequence then $l[0]$ is the first element, and $l[-1]$ is the last one.
\end{enumerate}

\subsection{Paper overview}
Section \ref{sec: Combinatorial embedding} introduces the concept of combinatorial embedding, a notion closely related to minors. Instead of working directly with minors, we use combinatorial embeddings to prove theorems \ref{thm: minors in hypercube} and \ref{thm: product graph minor-universality}. In Section \ref{sec: Lower bound - minors in the hypercube}, we establish the first part of theorem \ref{thm: product graph minor-universality} (which implies the lower bound of theorem \ref{thm: minors in hypercube}), treating proposition \ref{thm: box permutation decomposition} as a black box. The proof of proposition \ref{thm: box permutation decomposition} is provided in Section \ref{sec: Box permutation decomposition}. Finally, in Section \ref{sec: Upper bound - minors in the hypercube}, we complete the proof of theorem \ref{thm: minors in hypercube} by establishing the upper bound. This section also includes the proof of the second part of theorem \ref{thm: product graph minor-universality}, which is similar but requires more extensive computations.

\subsection{Acknowledgments}
Thanks to Michael Krivelevich for providing us with relevant references.
E.T. appreciates the support of the Adams Fellowship Program of the Israel Academy of Sciences and Humanities.

\section{Combinatorial embedding}
\label{sec: Combinatorial embedding}

\begin{definition}[combinatorial embedding]
Let $G$ and $Y$ be graphs. A \textbf{combinatorial embedding} of $G=(V,E)$ into $Y$ is a function 
$$f: V \cup E \to \bigcup_{n \in \N} (VY)^n$$
such that the image of each vertex is a vertex, and the image of an edge $e=uv\in EG$, denoted by $f_e$, is a walk with endpoints $fu, fv$. We call this walk a \textbf{road}. Moreover the followings hold:
\begin{enumerate}
    \item $f|_{V}$ is injective.
    \item "\textit{Roads do not intersect}" — for any two non-adjacent edges $e, t \in E$, we have $f_e \cap f_t = \emptyset.$
    \item "\textit{Vertices are isolated from roads}" — for any edge $e \in E$ and any vertex $v \notin e$, we have $fv \notin f_e.$
\end{enumerate}
\end{definition}
Combinatorial embeddings and minors are closely related concepts, as shown in the following two lemmas:
\begin{lemma}
\label{lemma: minor => comb emb}
Let \( H \) and \( G \) be graphs. If \( H \leq G \), then there exists a combinatorial embedding of \( H \) into \( G \).
\end{lemma}
To prove this, we begin with the following definition:  
A \emph{\textbf{model} of a graph} \( H \) \emph{in a graph} \( G \) is a function \( \mu \) that assigns to each vertex of \( H \) a vertex-disjoint connected subgraph of \( G \), such that if \( uv \in E(H) \), then some edge in \( G \) connects a vertex in \( \mu(u) \) to a vertex in \( \mu(v) \).

It is known that there exists a model of a graph \( H \) in a graph \( G \) if and only if \( H \leq G \) \cite[Lemma 1.2]{No}.
Thus, it suffices to prove that if there is a model of a graph \( H \) in \( G \), then there exists a combinatorial embedding of \( H \) into \( G \), indeed:

\begin{proof}
Assume that \( \mu \) is a model of \( H \) in \( G \). We aim to construct a combinatorial embedding \( f: H \to G \). For each vertex \( v \in V(H) \), define \( fv \) as some vertex within \( \mu(v) \). For each edge \( uv \in E(H) \), there exists an edge \( e \in E(G) \) that connects a vertex in \( \mu(u) \) to a vertex in \( \mu(v) \). Since \( \mu(u) \) and \( \mu(v) \) are connected, there is a walk from \( fu \) to \( fv \) that remains entirely within \( V(\mu(u)) \cup V(\mu(v)) \). Define \( f_e \) as this walk. This construction defines a combinatorial embedding, which is straightforward to verify.
\end{proof}

We now consider the other direction, namely, whether the existence of a combinatorial embedding of \( H \) into \( G \) implies that \( H \leq G \). Unfortunately, this is not true, as shown in the image.
\begin{figure}[ht]
    \centering
    \resizebox{0.3\textwidth}{!}{  
        \begin{tikzpicture}
            \filldraw[black] (0,0) circle (2pt) node[anchor=east] {};
            \filldraw[black] (2,0) circle (2pt) node[anchor=west] {};
            \filldraw[black] (1,1.73) circle (2pt) node[anchor=south] {};
            \draw (0,0) -- (2,0) -- (1,1.73) -- (0,0);

            \filldraw[black] (6,1) circle (2pt);
            \filldraw[black] (8,0) circle (2pt);
            \filldraw[black] (4,0) circle (2pt);
            \filldraw[black] (6,3) circle (2pt);
            \draw (6,1) -- (8,0);
            \draw (6,1) -- (4,0);
            \draw (6,1) -- (6,3);

            \draw [->,thick, bend left=40] (2.5,1) to (4.5,1);
        \end{tikzpicture}
    }
    \caption{There is a combinatorial embedding of the triangle into the star graph, but the triangle is not a minor of the star graph.}
\end{figure}
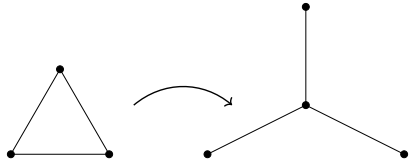
However, by slightly modifying \( H \), we can establish a related result that holds. The issue with converting a combinatorial embedding into a model is that, for two different edges sharing a common vertex, there is no guarantee about how their corresponding roads will behave. They may overlap excessively, as seen in the triangle example, which is problematic. To address this, we introduce the following definition to separate the roads and ensure that all edges are effectively non-adjacent.

\begin{definition}
For a graph \( G = (V, E) \), the \textbf{simple subdivision} of \( G \) is the graph \( G' = (V', E') \) where
\[
V' := V \cup \{(v, u) \mid v \sim u\}
\]
and
\[
E' := \{ \{v, (v, u)\} \mid v \sim u\} \cup \{\{(v, u), (u, v)\} \mid u \sim v\}.
\]
\end{definition}
In simple words, the simple subdivision of a graph is the graph obtained by taking the original graph and adding two new vertices "inside" each edge. Each such new vertex is related to one of the endpoints of the edge.

We can now state the following lemma:
\begin{lemma}
\label{lemma: comb emb of simple subdivision => minor}
Let \( H \) and \( G \) be graphs. If there is a combinatorial embedding of the simple subdivision of \( H \) into \( G \), then \( H \leq G \).
\end{lemma}

\begin{proof}
Let \( H' \) be the simple subdivision of \( H \), and let \( f: H' \to G \) be a combinatorial embedding. To prove \( H \leq G \), we will construct a model \( \mu \) of \( H \) in \( G \).

For each vertex \( v \in V(H) \), define its star set as
\[
S_v := \bigcup_{v \sim u} f_{\{v, (v, u)\}}.
\]
Notice that these star sets are pairwise disjoint since \( f \) is a combinatorial embedding.

For the edge \( e' = \{(u, v), (v, u)\} \in E(H') \), consider the road \( f_{e'} = (x_1, \ldots, x_n) \) where each \( x_i \) is a vertex in \( G \), \( x_1 = f((u, v)) \), and \( x_n = f((v, u)) \). Observe that the road starts at a point in \( S_u \) and ends at a point in \( S_v \). Along the way, it may enter and exit \( S_u \) and \( S_v \) arbitrarily. However, we can restrict a path segment \( (x_i, \ldots, x_j) \) so that \( x_i \in S_u \), \( x_j \in S_v \), and \( x_k \in (S_u \cup S_v)^C \) for any \( i < k < j \). Define the sets \( R_{u, v} := \{x_i\} \) and \( R_{v, u} := \{x_{i+1}, \ldots, x_j\} \) and note that
\begin{equation}
\label{eq1}
R_{u, v} \cap S_v = \emptyset \quad \text{and} \quad R_{v, u} \cap S_u = \emptyset.
\tag{$\ast$}
\end{equation}

For each vertex \( v \in V(H)\), define \( \mu(v) \) as the subgraph of \( G \) induced by the vertex set
\[
V_v := S_v \cup \bigcup_{w \sim v} R_{v, w}.
\]
The subgraph \( \mu(v) \) is connected because every vertex in it is connected to \( fv \). Additionally, for distinct vertices \( u, v \in V(H)\), we have \( V(\mu(u)) \cap V(\mu(v)) = V_u \cap V_v = \emptyset \), as shown by the following:
\begin{itemize}
    \item \( S_u \cap S_v = \emptyset \).
    \item For \( w \sim v \), we have \( S_u \cap R_{v, w} = \emptyset \):
    \begin{itemize}
        \item If \( w \neq u \), then \( R_{v, w} \subseteq f_{\{(v, w), (w, v)\}} \), a road that is disjoint from all roads forming \( S_u \).
        \item If \( w = u \), then \( S_u \cap R_{v, w} = S_u \cap R_{v, u} = \emptyset \) by (\ref{eq1}).
    \end{itemize}
    \item For \( w \sim u \) and \( w' \sim v \), \( R_{u, w} \cap R_{v, w'} = \emptyset \):
    \begin{itemize}
        \item If \( \{u, w\} \neq \{v, w'\} \), then \( R_{u, w} \subseteq f_{\{(u, w), (w, u)\}} \) and \( R_{v, w'} \subseteq f_{\{(v, w'), (w', v)\}} \), which are disjoint roads.
        \item If \( \{u, w\} = \{v, w'\} \), then \( w = v \) and \( w' = u \), so \( R_{u, w} \cap R_{v, w'} = R_{u, v} \cap R_{v, u} = \emptyset \) by definition.
    \end{itemize}
\end{itemize}
Thus, \( \mu \) is a model of \( H \) in \( G \).
\end{proof}

\section{Lower bound: minor-universality}
\label{sec: Lower bound - minors in the hypercube}
Here we prove the first part of theorem \ref{thm: product graph minor-universality}, which implies the
lower bound of theorem \ref{thm: minors in hypercube} - the assertion that \( Q_d \) is \( \Omega\left( \frac{2^d}{d} \right) \)-minor-universal. 
We present here the proof of the generalized theorem, as it not only generalizes the original result but also sheds light on the real essence of the proof, offering a more natural and comprehensive perspective.

Recall that the \emph{Cartesian product} of two graphs, \( G \) and \( H \), is the graph \( G \Box H \) with vertex set \( V(G) \times V(H) \), where vertices \( (g, h) \) and \( (g', h') \) are adjacent if either \( g = g' , h \sim h' \) or \( g \sim g', h = h' \). Observe that if \( H \leq G \), then \( H \Box L \leq G \Box L \) for any graph $L$.

For the rest of this section, fix \( k, n \in \mathbb{N} \), and let \( H_4, H_k, G_1, \ldots, G_n \) be a sequence of connected graphs such that \( |V(H_4)| = 4 \), \( |V(H_k)| = k \), and \( |V(G_i)| \leq k \) for all \( i \). Define \( V_i := V(G_i) \), and label the vertices of each of these graphs by $\{0,1,\ldots\}$. We also consider the cycle graph \( C_{6n-2}\), and also labeled its vertices by \( \{0, 1, \ldots\} \).

\begin{theorem}[lower bound of theorem \ref{thm: product graph minor-universality}]\label{thm: generalized lower bound}
The graph $$H_4\Box C_{6n-2}\Box H_k\Box G_1\Box \ldots \Box G_n$$ is $\frac{1}{16}|V_1|\cdots|V_n|$-minor-universal.
\end{theorem}

\subsection{Why this theorem is a generalization?}
Why is this theorem a generalization of the claim that \( Q_d \) is \( \Omega\left(\frac{2^d}{d}\right) \)-minor-universal? Let \( d \in \mathbb{N} \), and set \( n := d - \lfloor \log d \rfloor - 8 \). Applying the theorem with \( k = 2 \), \( H_4 = Q_2 \), \( H_k = Q_1 \), and \( G_i = Q_1 \), we conclude that the graph \( Q_2 \Box C_{6n-2} \Box Q_{n+1} \) is \( \Omega(2^n) \)-minor-universal. 

Since \( C_{6n-2} \leq Q_{\lceil \log(6n-2) \rceil} \), it follows that \( Q_{n + \lceil \log(6n-2) \rceil + 3} \) is \( \Omega(2^n) \)-minor-universal. Compute:
\[
n + \lceil \log(6n-2) \rceil + 3 < n + \log n + 7 < d - \log d - 7 + \log d + 7 = d.
\]
Thus, also \( Q_d \) is \( \Omega(2^n) \)-minor-universal. Additionally, since \( 2^n \geq 2^{-8} \frac{2^d}{d} \), we conclude that \( Q_d \) is \( \Omega\left(\frac{2^d}{d}\right) \)-minor-universal.

\subsection{Proof plan}
Before presenting the proof, we introduce additional notation to improve readability:
\[
\begin{aligned}
P_0 & := G_1 \Box \cdots \Box G_n, \\
P_1 & := H_k \Box P_0, \\
P_2 & := C_{6n-2} \Box P_1, \\
P_3 & := H_4 \Box P_2.
\end{aligned}
\]

Our goal is to show that \( P_3 \) has a lot of minors. By Lemma \ref{lemma: comb emb of simple subdivision => minor}, it suffices to show that a lot of graphs can be embedded into \( P_3 \). The original incentive was that a lot of graphs can be embedded (up to small errors) inside \( P_0 \) - the product of a sequence of connected graphs. The addition of \( H_4 \Box C_{6n-2} \Box H_k \) is an artifact that enable our embeddings. This artifact acts as a type of "coloring" that separates roads of the embedding.

The proof plan proceeds as follows: we start with a graph with a small number of edges that needs to be embedded into \( P_3 \). After reducing this problem to the problem of embedding a graph with a maximum degree \( \leq 3 \) and essentially same number of edges, we solve the reduced problem using the following three-step solution:

\begin{enumerate}
    \item \textbf{Piecewise embedding of matching graph}:We introduce a weaker notion of embedding, called a piecewise embedding. To enable the construction of such an embedding, we add \( H_k \) in front of \( P_0 \). Our goal is to find a piecewise embedding of a matching graph within \( P_1 = H_k \Box P_0 \), where the vertex positions are predetermined. This is achievable—using the \emph{box permutation decomposition}, we construct such a \((6n-3)\)-piecewise embedding.
    \item \textbf{From piecewise embedding to combinatorial embedding}: Adding \( C_{6n-2} \) in front of \( P_1 \), we convert a \( (6n-3) \)-piecewise embedding inside \( P_1 \) to a combinatorial embedding inside $P_2=C_{6n-2}\Box P_1$.
    \item \textbf{From matching graph to graph with maximal degree $\leq 3$}: Adding \( H_4 \) in front of \( P_2 \), we use the ability to embed matching graphs with given vertex positions into $P_2$ to construct a combinatorial embedding of graph with maximal degree $\leq 3$ into \( P_3 = H_4 \Box P_2 \).
\end{enumerate}

Now we present the proof itself:

\subsection{Reduction}
It's enough to prove
\begin{theorem}
\label{thm: reduced lower bound}
Every graph with maximal degree $\leq 3$ and $\leq \frac{1}{4}|V_1|\cdots|V_n|$ vertices is a minor of  $\fullGraph$.
\end{theorem}
Given a graph \( \G \) with no isolated vertices and \( m \leq \frac{1}{16}|V_1|\cdots|V_n| \) edges, we construct a new graph \(\G' \) by replacing each vertex \( v \) of \( \G \) with a binary tree having \(\deg(v)\) leaves and at most \( 2\deg(v) \) vertices in total. The binary trees are connected such that if \( u \) and \( v \) are neighbors in \( \G \), an edge is added between the leaf corresponding to \( u \) in the tree of \( v \) and the leaf corresponding to \( v \) in the tree of \( u \). 

Note that \( \G \) is a minor of \( \G' \). Furthermore, \( \G' \) has at most \( \sum_{v \in V(H)} 2\deg(v) = 4m \leq \frac{1}{4}|V_1|\cdots|V_n|\) vertices and a maximum degree of at most $3$. By the theorem, this implies \( \G' \leq \fullGraph \), and therefore \( \G \leq \fullGraph\). In conclusion, $\fullGraph$ is \(\frac{1}{16}|V_1|\cdots|V_n|\)-minor-universal, as we wanted.

To prove theorem \ref{thm: reduced lower bound}, we proceed by the planed three-steps solution:
\subsection{Step 1: Piecewise embedding of matching graph}
\begin{definition}[piecewise embedding]
For graphs \( \G \) and \( Y \) and \( N\in \N \), an \textbf{\( \bm{N }\)-piecewise embedding} from $\G$ to $Y$ is a pair
\[
(f, \{\gamma_e^i\}_{e \in E(\G), i \in [N]})
\]
where \( f: V(\G) \to V(Y) \) is an injective function, and each \( \gamma_e^i \) is a walk in \( Y \) such that \( \gamma_e := \gamma_e^1 \ast \dots \ast \gamma_e^N \) (concatenation of walks) is a walk with endpoints given by \( f(\text{endpoints}(e)) \). Additionally, for any two non-adjacent edges \( e, e' \in E(\G) \) and any \( i \in [N] \), we have:
\[
\gamma_e^i \cap \gamma_{e'}^i = \emptyset.
\]
\end{definition}
In this step we prove the following lemma:
\begin{lemma}\label{lemma: step 1}
Let $\G$ be a matching graph, if  $f:V(\G)\ra  V(P_0)$ is an injective function, then there is a $(6n-3)$-piecewise embedding $(\tilde f, \{\gamma_e^i\}_{e \in E(\G), i \in [6n-3]})$ from $\G$ to $P_1$ such that $\tilde fv=(0,fv)$ for every $v\in V(\G)$.
\end{lemma}
\begin{proof}
Define $N:=6n-3$. Let \( \G = (V, E) \) be a matching graph, and let \(
f: V\ra  V(P_0)\) be an injective function. We will find a set \( \{\gamma_e^i\}_{e \in E, i \in [N]} \) such that \( (\tilde f, \{\gamma_e^i\}_{e \in E, i \in [N]}) \) forms a \( N \)-piecewise embedding of \( \G \) into \( P_1 \) where  $\tilde fv=(0,fv)$ for every $v\in V(\G)$.

To construct the desired set \( \{\gamma_e^i\}_{e \in E, i \in [N]} \), we use proposition \ref{thm: box permutation decomposition} (box permutation decomposition) by finding a relevant permutation for the problem. Since \( \G \) is a matching graph, there is a permutation $\s\in \text{Per}(V(P_0))=\text{Per}(V_1\times ...\times V_n)$ such that if $u\sim v$ then $\s(fu)=fv$.

Using proposition \ref{thm: box permutation decomposition}, we obtain a set of permutations \( \s_1,...,\s_{2n-1}\in \text{Per}(V_1\times ...\times V_n) \) such that
\[
\sigma = \sigma_1 \circ \dots \circ \sigma_{2n-1}
\]
where each \( \sigma_i \) is a one-dimensional permutation. This implies that for every \(i \in [2n-1]\), there exists a coordinate \(j_i\) such that, for every \(\vec{v} \in V_1 \times \dots \times V_n\), the vectors \(\vec{v}\) and \(\sigma_i(\vec{v})\) can differ only in the \(j_i\) coordinate. Since \(G_{j_i}\) is connected, there exists a path in \(G_1 \Box \dots \Box G_n\) from \(\vec{v}\) to \(\sigma_i(\vec{v})\) that remains constant on all coordinates except maybe \(j_i\). Let us denote this path by \(P_{i,\vec{v}}\).

Next, we define \( \{\gamma_e^i\}_{e \in E, i \in [N]} \). First define  \( \tau_i := \sigma_1 \circ \dots \circ \sigma_{i} \) for $i\in [ 2n-1]$ and $\tau_0:=id$. For an edge \( e \in E \) we choose an endpoint $u\in e$ and define \( \{\gamma_e^i\}_{i \in [k]} \) in trios: For each \( i \in [2n-1] \) define the path $\gamma^{3i-2}_e$ to be $$(0,\tau_{i-1}(fu))\stackrel{0\rightsquigarrow[\tau_{i-1}(fu)]_{j_i}}{\rightsquigarrow }([\tau_{i-1}(fu)]_{j_i},\tau_{i-1}(fu)),$$
the path $\g_e^{3i-1}$ to be 
$$([\tau_{i-1}(fu)]_{j_i},\tau_{i-1}(fu))\stackrel{P_{i,\tau_{i-1}(fu)}}{\rightsquigarrow } ([\tau_{i-1}(fu)]_{j_i},\tau_i(fu)),$$
and the path $\g_e^{3i}$ to be $$([\tau_{i-1}(fu)]_{j_i},\tau_i(fu))\stackrel{[\tau_{i-1}(fu)]_{j_i}\rightsquigarrow0}{\rightsquigarrow }(0,\tau_i(fu)).$$

It is straightforward to see that \( \{\gamma_e^i\}_{i \in [N]} \) was chosen so that \( \gamma_e := \gamma_e^1 \ast \dots \ast \gamma_e^{N} \) forms a walk from \( \tilde{f}u \) to \( \tilde{f}v \). We now verify that this indeed constitutes an \( N \)-piecewise embedding:

\begin{enumerate}
    \item \( \tilde{f} \) is injective because \( f \) is injective.
    \item For any two non-adjacent edges \( e, e' \in E \), let \( u\in e, u'\in e' \) be the vertices such that \( \gamma_e^1[0] = (0,fu) \) and \( \gamma_{e'}^1[0] = (0,fu') \). Let \( j \in [2n-1] \); we claim that \( \gamma_e^j \cap \gamma_{e'}^j = \emptyset \). Indeed, if \( j \in 3\mathbb{Z} + \{0,1\} \), it is easy to observe that \( \gamma_e^j \cap \gamma_{e'}^j = \emptyset \). Now, assume \( j \equiv -1 \pmod{3} \) and set \( i := (j+1)/3 \). If $\gamma_e^j$ and $\gamma_{e'}^j$ intersect, then $[ \tau_{i-1}(fu) ]_{j_i} = [ \tau_{i-1}(fu') ]_{j_i}$, and the paths $P_{i,\tau_{i-1}(fu)}$ and $P_{i,\tau_{i-1}(fu')}$ also intersect. Since every point in $P_{i,\tau_{i-1}(fu)}$ can differ from $\tau_{i-1}(fu)$ only at the coordinate $j_i$, and the same is true for $u'$, it follows that $\tau_{i-1}(fu)$ and $\tau_{i-1}(fu')$ can differ only at $j_i$. However, as $\tau_{i-1}(fu) \neq \tau_{i-1}(fu')$, we must have $[ \tau_{i-1}(fu) ]_{j_i} \neq [ \tau_{i-1}(fu') ]_{j_i}$, leading to a contradiction.
    \end{enumerate}
\end{proof}

\subsection{Step 2: From piecewise embedding to combinatorial embedding}
In this step we prove:
\begin{lemma}\label{lemma: step 2}
Let $\G$ be a matching graph, if  $f:V(\G)\ra  \{(0,0)\}\times V(P_0)$ is an injective function, then there is a combinatorial embedding embedding $\tilde f:\G\ra P_2$ such that $\tilde f|_{V(\G)}=f$.
\end{lemma}
The following claim together with lemma \ref{lemma: step 1} immediately prove this.
\begin{lemma}
For a graph \( \G \) and \( N\in \N \), let \( (f, \{\gamma_e^i\}_{e \in E(\G), i \in [N]}) \) be a \( N \)-piecewise embedding of $\G$ into some graph \( Y \). There exists a combinatorial embedding \( \tilde{f}: G \to C_{N+1}\Box Y \) where $\tilde fv=(0,fv)$ for every $v\in V$.
\end{lemma}
\begin{proof}
Denote $\G=(V,E)$. For \( v \in V \), define
\[
\tilde{f}v := (0, fv).
\]
For \( e \in E \), define \( \tilde{f}_e \) as follows:
\begin{align*}
(0, \gamma_e^1[0]) &\rightarrow (1, \gamma_e^1[0]) \stackrel{\gamma_e^1}{\rightsquigarrow} (1, \gamma_e^1[-1]) \\
&\rightarrow (2, \gamma_e^2[0]) \stackrel{\gamma_e^2}{\rightsquigarrow} (2, \gamma_e^2[-1]) \\
&\vdots \\
&\rightarrow (N, \gamma_e^N[0]) \stackrel{\gamma_e^N}{\rightsquigarrow} (N, \gamma_e^N[-1]) \rightarrow (0, \gamma_e^N[-1]).
\end{align*}

This is indeed a walk with endpoints $\tilde f(\text{endpoints}(e))$. Now we verify that \( \tilde{f} \) is a combinatorial embedding:
\begin{enumerate}
    \item \( \tilde{f}|_V \) is injective since \( f|_V \) is injective.
    \item Let \( e, e' \in E \) be two non-adjacent edges. Assume there is \( (\xi_1, \xi_2) \in \tilde{f}_e \cap \tilde{f}_{e'} \), so \( \xi_1=i \) for some \( 0 \leq i \leq k \). 
        \begin{itemize}
            \item If \( i =0 \), then \( \xi_2 \in f(\text{endpoints}(e)) \cap f(\text{endpoints}(e')) = \emptyset \), a contradiction.
            \item If \( 1 \leq i \leq N \), then \( \xi_2 \in \gamma^i_e \cap \gamma^i_{e'} = \emptyset \), another contradiction.
        \end{itemize}
    \item Let \( e \) be an edge and \( v \) a vertex not in it. Since \( \tilde{f}v = (0, fv) \), if \( \tilde{f}v \in \tilde{f}_e \), then \( fv \in f(\text{endpoints}(e)) \), which is impossible as \( f|_V \) is injective.
\end{enumerate}
\end{proof}
\subsection{Step 3: From matching graph to graph with maximal degree $\leq 3$ }
In this step we conclude theorem \ref{thm: reduced lower bound}:
\begin{proof}
Let $\G$ be a graph with maximal degree $\leq 3$ and $\leq \frac{1}{4}|V_1|\cdots|V_n|$ vertices. Denote by $\G'=(V,E)$ its simple subdivision, we will construct a combinatorial embedding of $\G'$ into $\fullGraph$ which will imply $\G\leq \fullGraph$, as needed.

Notice that $|V|\leq 4|V(\G)|\leq |V_1|\cdots|V_n|$, so there is an injective function $f:V\ra \{(0,0)\}\times V_1\times \ldots \times V_n$. By \textit{Vizing's Theorem}, we can partition \( E \) into \( 4 \) disjoint subsets \( E_0, \dots, E_3 \) such that for each \( i \), the graph \( (V, E_i) \) is a matching graph. By lemma \ref{lemma: step 2}, for any $i\in \{0,1,2,3\}$, there is a combinatorial embedding \( f^{(i)} \) of \( (V, E_i) \) into \( \GraphWithTwoAdditionals \) that extends \( f \).

We construct a combinatorial embedding $\tilde f:\G'\ra \fullGraph$: For \( v \in V \), define \( \tilde fv := (0, fv) \). For \( i \in \{0,1,2,3\} \) and \( e=uv \in E_i \), define \(\tilde {f}_e \) as
\[
(0, fu) \stackrel{0\rightsquigarrow i}{\rightsquigarrow} (i, fu) \stackrel{f^{(i)}_e}{\rightsquigarrow} (i, fv) \stackrel{i\rightsquigarrow 0}{\rightsquigarrow} (0, fv).
\]

Now we verify that \(\tilde  f \) is a combinatorial embedding:
\begin{itemize}
    \item \( \tilde {f}|_V \) is injective since \( f|_V \) is injective.
    \item Let \( e, e' \in E \) be two non-adjacent edges, and let \( i \) and \( i' \) be the indices such that \( e \in E_i \) and \( e' \in E_{i'} \). We aim to show that \( \tilde{f}_e \cap \tilde{f}_{e'} = \emptyset \). By projecting these walks onto the second component, we obtain \( f^{(i)}_e \) and \( f^{(i')}_{e'} \). Since \( f_e^{(i)} \cap f(\text{endpoints}(e')) = \emptyset \) and vice versa, we conclude that any intersection between \( \tilde{f}_e \) and \( \tilde{f}_{e'} \) could only happen in the "middle phase" (the second arrow in the definition of $\tilde{f}_e$). If \( i \neq i' \), no intersection can occur. If \( i = i' \), an intersection is also impossible because \( f^{(i)} \) is a combinatorial embedding, so \( f^{(i)}_e \cap f^{(i)}_{e'} = \emptyset \).
    \item Let \( e \) be an edge and \( v \) a vertex not in \( e \). There is \( i \) such that \( e \in E_i \). We need to show that \( \tilde {f}v \notin \tilde{f}_e \), which follows immediately from \( fv \notin f^{(i)}_e \).
\end{itemize}
\end{proof}

\section{Box permutation decomposition}
\label{sec: Box permutation decomposition}
Here we prove proposition \ref{thm: box permutation decomposition} about decomposition of permutation of a box. We will first prove a claim called the \textit{well-dispersed lemma} which will help us prove the upper bound of the proposition (that any permutation can be decomposed into $2d-1$ one-dimensional permutations). We finish this section by proving the lower bound of the proposition (that there exists a permutation such that any decomposition into one-dimensional permutations requires $2d-1$ elements).

\begin{definition}
Let $d\in \N$ and $n_1,...,n_d\in \N$. Define $X:=[n_1]\times ...\times [n_d]$. A permutation $\s:X\ra X$ is called \textbf{$\bm i$-permutation} if it affects only the $i$-coordinate.
\end{definition}

\subsection{Well-dispersed Lemma}
\begin{definition}
A matrix \( A \) is \textbf{well-dispersed} if each column is a vector with all distinct components.
\end{definition}
\begin{lemma}[well-dispersed Lemma]
Let \( L \) be a finite set of labels, and let \( m := |L| \) be its size. Suppose \( A \in M_{m \times n}(L) \) is a matrix such that each label \( \ell \in L \) appears in exactly \( n \) entries of \( A \). Then, there exist \( m \) permutations \( \sigma_1, \dots, \sigma_m \in S_n \) such that the matrix \( \tilde{A} := (A_{i\sigma_i(j)})_{i,j} \) is well-dispersed.
\end{lemma}
\begin{figure}[h] 
    \centering
    \[
\begin{array}{c@{\hspace{2cm}}c}
 A &  \tilde{A} \\
\begin{array}{|c|c|c|c|}
\hline
1 & 3 & 5 & 5 \\ \hline
6 & 4 & 2 & 2 \\ \hline
5 & 2 & 1 & 3 \\ \hline
4 & 1 & 6 & 3 \\ \hline
1 & 2 & 4 & 3 \\ \hline
4 & 6 & 6 & 5 \\ \hline
\end{array}
&
\begin{array}{|c|c|c|c|}
\hline
1 & 5 & 5 & 3 \\ \hline
6 & 4 & 2 & 2 \\ \hline
5 & 2 & 3 & 1 \\ \hline
3 & 1 & 4 & 6 \\ \hline
2 & 3 & 1 & 4 \\ \hline
4 & 6 & 6 & 5 \\ \hline
\end{array}
\end{array}
\]

\begin{tikzpicture}[overlay, remember picture]
\foreach \i in {0, 1, 2, 3, 4, 5} {
    \draw[->][blue] (-0.9, 2.78-0.46*\i) -- (0.9, 2.78-0.46*\i);
}
\end{tikzpicture}
    \caption{Well-dispersed lemma example.}
    \label{fig:example-image}
\end{figure}
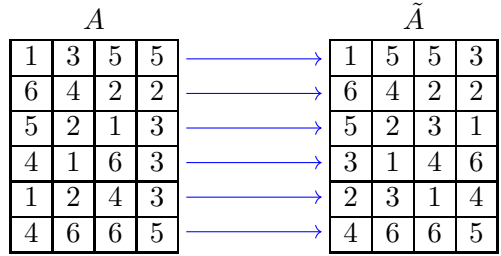

\begin{proof}
Observe that since the label set \( L \) has size \( m \), a matrix in \( M_{m \times n}(L) \) is well-dispersed if and only if each column contains all the labels.

We prove the statement by induction on the number of columns. When there is only one column, each label appears exactly once, making \( A \) well-dispersed. Now, assume \( n > 1 \):

If the first column of \( A \) contains all the labels, the result follows by the induction hypothesis. Otherwise, assume there exists a proper subset \( L' \subsetneq L \), representing the set of all labels in the first column of \( A \). It suffices to find permutations \( \sigma_1, \dots, \sigma_m \in S_n \) such that the first column of the matrix \( (A_{i\sigma_i(j)})_{i,j} \) will contain \( |L'| + 1 \) labels.

Since the first column of \( A \) does not contain all the labels, there must be a label \( \ell_0 \in L \) that appears twice in this column. Let \( I_0 := \{i_0, i_1\} \) denote two indices such that \( A_{i_0 1} = A_{i_1 1} = \ell_0 \), and define \( L_0 := \{A_{ij} \mid (i,j) \in I_0 \times [n]\} \) to be the set of labels in the rows indexed by $I_0$. For \( k \in [m] \), define
\[
I_k := \{i \in [m] \mid A_{i1} \in L_{k-1}\}
\]
\[
L_k := \{A_{ij} \mid (i,j) \in I_k \times [n]\}
\]
Notice that \( I_k \subseteq I_{k+1} \) and \( L_k \subseteq L_{k+1} \).

We claim that each label in \( L_m \) can be reached. Specifically, for every \( (a, b) \in I_m \times [n] \), there exists a way to permute elements within each row of \( A \) such that the resulting matrix has labels \( L' \cup \{A_{ab}\} \) in the first column. Given this, it remains to show that \( L_m \) contains at least one label from \( L \setminus L' \). These  claims will be proven in the lemmas below.
\end{proof}
\begin{lemma}
For every \( 0 \leq k \leq m \) and every entry \( (a, b) \in I_k \times [n] \), there exist permutations \( \{\tau_i \in S_n\}_{i \in I_k} \) such that \( \{A_{i \tau_i(1)}\}_{i \in I_k} = \{A_{i1}\}_{i \in I_k} \cup \{A_{ab}\} \). In other words, after applying these permutations to the rows indexed by \( I_k \), the new first column will contain the label \( A_{ab} \) without losing any existing labels.
\end{lemma}
\begin{proof}
We prove this by induction on \( k \). For the base case, let \( k = 0 \) and \( (a, b) \in \{i_0, i_1\} \times [n] \). Without loss of generality, assume \( a = i_0 \). Then, set \( \tau_{i_0} := (1b) \) and \( \tau_{i_1} := \text{id} \). Indeed, we have \( \{A_{i_0 \tau_{i_0}(1)}, A_{i_1 \tau_{i_1}(1)}\} = \{A_{i_0 b}, A_{i_1 1}\} = \{A_{i_0 1}, A_{i_1 1}\} \cup \{A_{ab}\} \).

Now, consider the case \( k > 0 \) and \( (a, b) \in I_k \times [n] \). If \( a \in I_{k-1} \), we conclude by the induction hypothesis. Otherwise, assume \( a \in I_k \setminus I_{k-1} \). Our goal is to swap the $1$st and \( b \)-th components of the \( a \)-th row so that the label \( A_{ab} \) appears in the first column. To achieve this without losing the label \( A_{a1} \), we use the induction hypothesis to permute the components in the rows  indexed by \( I_{k-1} \)so that the first column contains \( A_{a1} \) twice!

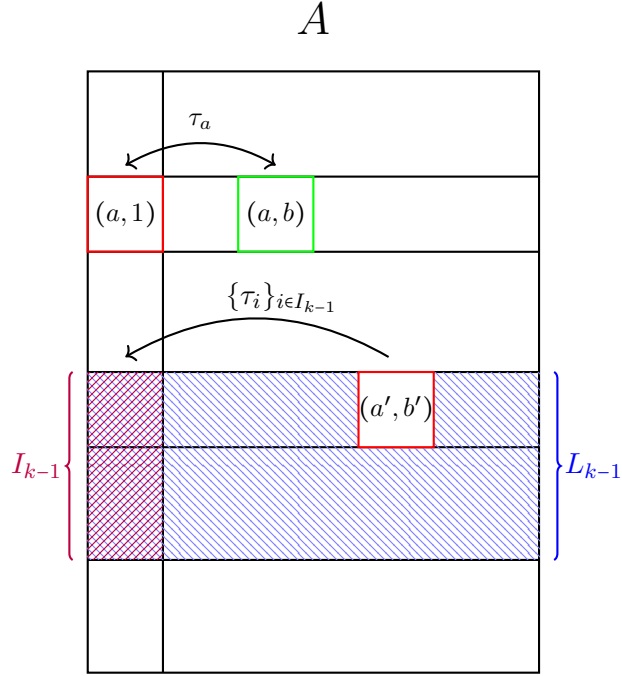
\begin{figure}[H] 
    \centering
\begin{tikzpicture}

\draw[thick] (0,0) rectangle (6,8);

\draw[thick] (0,6.6) -- (6,6.6); 
\draw[thick] (0,5.6) -- (6,5.6); 
\draw[thick] (0,4) -- (6,4); 
\draw[thick] (0,3) -- (6,3); 
\draw[thick] (0,1.5) -- (6,1.5); 

\draw[thick] (1,0) -- (1,8); 

\node at (3,8.7) {\LARGE $A$};

\draw[decorate, decoration={brace, mirror}, thick, purple] (-0.2,4) -- (-0.2,1.5) node[midway, left] {\textcolor{purple}{$I_{k-1}$}};
\draw[decorate, decoration={brace}, thick, blue] (6.2,4) -- (6.2,1.5) node[midway, right] {\textcolor{blue}{$L_{k-1}$}};

\draw[thick, red] (0,6.6) rectangle (1,5.6);
\node at (0.5,6.1) {\small $(a,1)$};

\draw[thick, green] (2,6.6) rectangle (3,5.6);
\node at (2.5,6.1) {\small $(a,b)$};

\fill[pattern=north west lines, pattern color=blue!50] (0,4) rectangle (6,1.5);

\fill[white] (3.6,3) rectangle (4.6,4);

\draw[thick, red] (3.6,3) rectangle (4.6,4);
\node at (4.1,3.5) {\small $(a',b')$};

\fill[pattern=north east lines, pattern color=purple] (0,4) rectangle (1,1.5);

\draw[<->, thick] (0.5,6.75) to[out=30, in=150] (2.5,6.75);
\node[above] at (1.5,7.1) {\small $\tau_a$};

\draw[<-, thick] (0.5,4.2) to[out=30, in=150] (4,4.2);
\node[right] at (1.7,5) { $\{\tau_i\}_{i \in I_{k-1}}$};

\end{tikzpicture}
    
    \caption{Well-dispersed lemma proof strategy. }
    \label{fig:example-image}
\end{figure}

By the definition of \( I_k \), there exist \( a' \in I_{k-1} \) and \( b' \in [n] \) such that \( A_{a1} = A_{a'b'} \). By the induction hypothesis, there are permutations \( \{\tau_i \in S_n\}_{i \in I_{k-1}} \) such that
\[
\tag{$\ast$}
\{A_{i \tau_i(1)}\}_{i \in I_{k-1}} = \{A_{i1}\}_{i \in I_{k-1}} \cup \{A_{a'b'}\}.
\]

Now, define \( \tau_a := (1b) \), and for \( i \in I_k \setminus (I_{k-1} \cup \{a\}) \), set \( \tau_i := \text{id} \).
This construction gives the desired result, as follows:
\begin{align*}
\{A_{i \tau_i(1)}\}_{i \in I_k} &= \{A_{i \tau_i(1)}\}_{i \in I_{k-1}} \cup \{A_{i \tau_i(1)}\}_{i \in I_k \setminus (I_{k-1} \cup \{a\})} \cup \{A_{a \tau_a(1)}\} \\
\overset{\scriptscriptstyle (\ast)}&{=}\{A_{i1}\}_{i \in I_{k-1}} \cup \{A_{a'b'}\} \cup \{A_{i 1}\}_{i \in I_k \setminus (I_{k-1} \cup \{a\})} \cup \{A_{a b}\}  \\
\overset{\scriptscriptstyle A_{a1} = A_{a'b'}}&{=} \{A_{i1}\}_{i \in I_k} \cup \{A_{ab}\}.
\end{align*}

\end{proof}

\begin{lemma}
\( L_m \not\subseteq L' \).
\end{lemma}
\begin{proof}
We will prove by induction on \( k \in [m] \) that if \( L_{k-1} \subseteq L' \), then \( |I_k| > k \).

For the base case \( k = 1 \), we have \( |I_1| \geq |I_0| > 1 \). For \( k > 1 \), assume that \( L_{k-1} \subseteq L' \). Then \( L_{k-2} \subseteq L' \) as well. By the induction hypothesis, we have \( |I_{k-1}| \geq k \). The set \( L_{k-1} \) consists of labels of \( |I_{k-1}| \cdot n \) entries of \( A \), with each label appearing at most \( n \) times, so \( |L_{k-1}| \geq |I_{k-1}| \geq k \). Since \( L_{k-1} \subseteq L' \), each label in \( L_{k-1} \) appears in \( \{A_{i1}\}_{i \in I_k} \). Therefore, \( |\{A_{i1}\}_{i \in I_k}| \geq k \). However, since \( A_{i_0 1} = A_{i_1 1} \), we must have \( |I_k| > k \).
\end{proof}
\subsection{Upper bound - box permutation decomposition}
Now are ready to prove the upper bound of proposition \ref{thm: box permutation decomposition}:
\begin{proposition}[upper bound of proposition \ref{thm: box permutation decomposition}]
Let \( n_1, \dots, n_d \in \mathbb{N} \) and define \( X := [n_1] \times \dots \times [n_d] \). Every permutation \( \sigma: X \to X \) can be written as
\[
\sigma = \sigma_d \circ \dots \circ \sigma_2 \circ \sigma_1 \circ \sigma'_2 \circ \dots \circ \sigma'_d,
\]
where \( \sigma_i \) and \( \sigma'_i \) are \( i \)-permutations.
\end{proposition}
\begin{proof}
The case $d=1$ is trivial, in this case $\s$ is already a $1$-permutation. We will prove only the case $d=2$ as it implies the case $d>2$:
$\s$ is a permutation of $[n_1]\times ... \times [n_d]$, it can be also seen as permutation of $Y\times [n_d]$ for $Y=[n_1]\times ... \times [n_{d-1}]$. Hence, given the case of $d=2$ we know that $\s=\s_d\circ \tau\circ \s_d'$
where $\s_d, \s_d'$ are $d$-permutations and $\tau$ is a permutation that preserve the $d$-coordinate. Hence for every $i\in [n_d]$, the function $\tau_i:=\tau|_{Y\times \{i\}}$ is a permutation. So, by induction, there is a decomposition  $$\tau_i = \sigma_{d-1,i}\circ...\circ \sigma_{2,i}\circ \sigma_{1,i}\circ \sigma'_{2,i}\circ...\circ \sigma'_{d-1,i}$$
where $\sigma_{j,i},\sigma'_{j,i}$ are $j$-permutations in $Y\times \{i\}\ra Y\times \{i\}$ for every $j\in [d-1]$.\\
Fix $j\in [d-1]$, we have the $j$-permutations $\s_{j,1},...,\s_{j,n_d}$ of  $Y\times \{1\},...,Y\times \{n_d\}$ respectively, we want 
to unify all of them to be a one $j$-permutation of $X$, hence we define $\s_j:X\ra X$ as $x\mapsto \s_{j,x_d}(x)$, and similarly define $\s_j'$. Since $\tau(x)=\tau_{x_d}(x)$ we get $\tau =\sigma_{d-1}\circ...\circ \sigma_{2}\circ \sigma_{1}\circ \sigma'_{2}\circ...\circ \sigma'_{d-1}$ as needed.

Now we prove the $d=2$ case, so assume $X=[m]\times [n]$. Define the matrix $A:=(\s^{-1}(i,j))_{ij}$, notice that we need to find $2$-permutations $\s_2,\s_2'\in Per(X)$ and $1$-permutation $\s_1\in Per(X)$ such that $A_{\s_2\circ\s_1\circ \s_2'(i,j)}=(i,j)$ for every $(i,j)\in [m]\times [n]$. Since each number $i \in [m]$ appears as the first component in exactly $n$ entries of $A$, we can use the Well-dispersed lemma to permute the elements in each row so that every column contains all possible first components. After this, we can rearrange the elements within each column so that the first components of the columns are in the order $(1, \dots, m)$. At this stage, for each $i \in [m]$, the labels in the $i$-th row all start with the first component $i$. Finally, we can permute the elements within each row so that the $i$-th row becomes $(i, 1), \dots, (i, n)$, completing the proof.

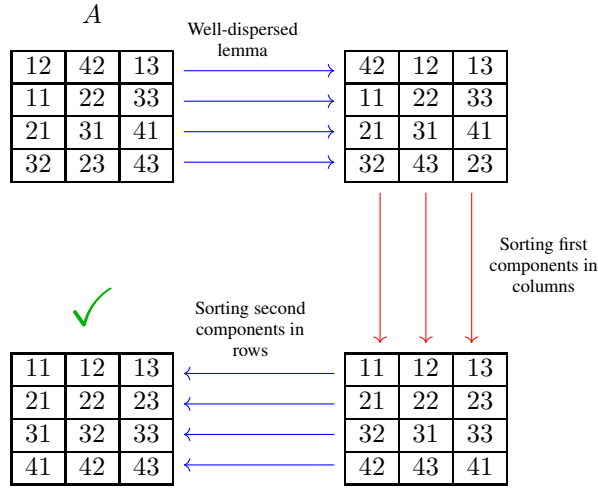
\begin{figure}[H] 
    \centering
    \begin{tikzpicture}[every node/.style={font=\small}, node distance=1cm]

\node (A) at (0, 0) {
    $\begin{array}{|c|c|c|}
    \hline
    12 & 42 & 13 \\ \hline
    11 & 22 & 33 \\ \hline
    21 & 31 & 41 \\ \hline
    32 & 23 & 43 \\ \hline
    \end{array}$
};

\node[above=0.1cm of A] {\textbf{$A$}};

\node (B) [right=2cm of A] {
    $\begin{array}{|c|c|c|}
    \hline
    42 & 12 & 13 \\ \hline
    11 & 22 & 33 \\ \hline
    21 & 31 & 41 \\ \hline
    32 & 43 & 23 \\ \hline
    \end{array}$
};

\node at (2,1) {\tiny \begin{tabular}{c} Well-dispersed \\ lemma \end{tabular}};

\node at (6,-2) {\tiny \begin{tabular}{c} Sorting first \\ components in \\ columns \end{tabular}};

\node at (2.1,-2.85) {\tiny \begin{tabular}{c} Sorting second \\ components in \\ rows \end{tabular}};

\foreach \i in {0.8, 0.6, 0.4, 0.2} {
    \draw[->, blue] ($(A.south east)!\i!(A.north east)$) -- ($(B.south west)!\i!(B.north west)$);
}

\node (C) [below=2cm of A] {
    $\begin{array}{|c|c|c|}
    \hline
    11 & 12 & 13 \\ \hline
    21 & 22 & 23 \\ \hline
    31 & 32 & 33 \\ \hline
    41 & 42 & 43 \\ \hline
    \end{array}$
};

\node[above=0.1cm of C] {\textcolor{green!70!black}{\huge $\checkmark$}};

\node (D) [right=2cm of C] {
    $\begin{array}{|c|c|c|}
    \hline
    11 & 12 & 13 \\ \hline
    21 & 22 & 23 \\ \hline
    32 & 31 & 33 \\ \hline
    42 & 43 & 41 \\ \hline
    \end{array}$
};

\foreach \i in {0.75, 0.5, 0.25} {
    \draw[->, red] ($(B.south west)!\i!(B.south east)$) -- ($(D.north west)!\i!(D.north east)$);
}

\foreach \i in {0.8, 0.6, 0.4, 0.2} {
    \draw[<-, blue] ($(C.south east)!\i!(C.north east)$) -- ($(D.south west)!\i!(D.north west)$);
}

\end{tikzpicture}
    \caption{$d=2$ strategy. }
    \label{fig:example-image}
\end{figure}
\end{proof}
\subsection{Lower bound: box permutation decomposition}
\begin{proposition}[lower bound of proposition \ref{thm: box permutation decomposition}]
 Let \( d \in \mathbb{N} \) and \( n_1, \ldots, n_d \geq 2 \). Define \( X := [n_1] \times \cdots \times [n_d] \). There exists a permutation \( \sigma: X \to X \) such that, for any expression of \( \sigma \) as a composition \( \sigma = \sigma_1 \circ \cdots \circ \sigma_r \) where each \( \sigma_i \) is a one-dimensional permutation, we have \( r \geq 2d - 1 \).   
\end{proposition}
\begin{proof}
Let \( \sigma \) be the permutation that swaps \( (1, \ldots, 1) \) and \( (2, \ldots, 2) \). Suppose \( \sigma = \sigma_1 \circ \cdots \circ \sigma_r \), where each \( \sigma_i \) is a non-identity one-dimensional permutation. Since \( \sigma \) swaps \( (1, \ldots, 1) \) and \( (2, \ldots, 2) \), there must be at least one \( i \)-permutation in this decomposition for every \( i \in [d] \). If for each \( i \in [d] \) there are at least two \( i \)-permutations in the decomposition, then there are at least \( 2d \) permutations, and we are done. Assume instead that there exists \( i \in [d] \) such that there is exactly one \( i \)-permutation in the decomposition, say \( \sigma_j \).

Define \( \tau' = \sigma_1 \circ \cdots \circ \sigma_{j-1} \) and \( \tau = \sigma_{j+1} \circ \cdots \circ \sigma_r \). The permutation \( \sigma_j \) is the \( i \)-permutation that swaps \( \tau(1, \ldots, 1) \) and \( \tau(2, \ldots, 2) \). Since, for any \( p \in X \setminus \{(1, \ldots, 1), (2, \ldots, 2)\} \), we have \( \sigma_j(\tau(p)) = \tau(p) \). This follows because \(
[\sigma_j(\tau(p))]_i = [\tau'(\sigma_j(\tau(p)))]_i = [\sigma(p)]_i = p_i = [\tau(p)]_i.
\)

Since \( \sigma_j \) swaps \( \tau(1, \ldots, 1) \) and \( \tau(2, \ldots, 2) \) and is an \( i \)-permutation, \( \tau(1, \ldots, 1) \) and \( \tau(2, \ldots, 2) \) must agree in \( d-1 \) coordinates. However, \( (1, \ldots, 1) \) and \( (2, \ldots, 2) \) disagree in all coordinates. Thus, \( \tau \) must consist of at least \( d-1 \) one-dimensional permutations. By the same reasoning, \( \tau' \) also consists of at least \( d-1 \) one-dimensional permutations. Consequently, the original decomposition consists of at least \( 2(d-1) + 1 = 2d - 1 \) one-dimensional permutations.

\end{proof}

\section{Upper bound: minor-universality}
\label{sec: Upper bound - minors in the hypercube}
We begin by proving the upper bound of theorem \ref{thm: minors in hypercube}, as its proof captures the core idea of the generalized version while being simpler and more elegant. In the next subsection, we will address the generalized version, corresponding to the second part of theorem \ref{thm: product graph minor-universality}.

\begin{theorem}[upper bound of theorem \ref{thm: minors in hypercube}]\label{thm: upper bound}
There is an absolute constant $C$ such that $Q_d$ is not $\frac{C2^d}{\sqrt d}$-minor-universal.
\end{theorem}

The idea is to utilize an expander graph as one that is hard to embed. This difficulty arises from examining how easily $Q_d$ can be separated, drawing inspiration from the concept of the separation profile introduced in \cite{BST}.

Since we will use expander graphs to prove this theorem, we provide here the precise definition of the Cheeger constant as follows:

\begin{definition}
Let \( G = (V, E) \) be a graph, and let \( U, W \subseteq V \) be disjoint subsets. We define \( \mathcal{E}(U, W) := \{ e \in E \mid |e \cap U| = 1, |e \cap W| = 1 \} \). The Cheeger constant of \( G \) is
\[
h(G) := \inf_{\emptyset \ne U \subseteq V, |U| \leq \frac{1}{2} |V|} \frac{|\mathcal{E}(U, U^C)|}{|U|}.
\]
\end{definition}

\begin{proof}
Let \( \{ G_i \}_{i\in 2\N} \) be a family of expander graphs (without isolated vertices) with Cheeger constant at least \( h \) (for some constant \( 0< h\leq  1 \)) and maximal degree at most \( 3 \), such that \( |V(G_i)| = i \). \footnote{It is well known that such a family exists, see e.g. \cite[Theorem 4.1.1]{Ko}.} 

Define the constant \( C := \frac{10 \sqrt{2} e}{\sqrt{\pi} h} \), let \( d \in \mathbb{N} \) and choose some even number $n$ such that \(C\frac{2^d}{\sqrt{d}} \leq n \leq 2C  \frac{2^d}{\sqrt{d}} \). Consider the graph \( G := G_n \) and write \( G = (V, E) \). This graph has at most $1.5n \leq 3C\frac{2^d}{\sqrt d}$ edges and it is not a minor of $Q_d$, as needed. We will show that $G\not\leq Q_d$ by showing that if \( G \leq Q_d \), then \( C\frac{2^d}{\sqrt{d}} > n \), which is a contradiction.

Assume that \( G \leq Q_d \). Then, by lemma \ref{lemma: minor => comb emb}, we have a combinatorial embedding \( f : G \rightarrow Q_d \). We will partition \(V(Q_d)=\{0,1\}^d\) into two parts using a relatively small cut \( S \) such that each part contains approximately half of the vertices of \( f(V) \). Since \( G \) has a large Cheeger constant, there are many roads of \( f \) that connect vertices of \( f(V) \) from one part, to vertices of \( f(V) \) from the other part. A large subset of these roads is pairwise disjoint, meaning that each road must pass through a distinct point in $S$, which implies that $S$ is large.  Since \( S \) is small relative to the cube \( Q_d \), it follows that the cube itself must be large, which means that \( \frac{2^d}{\sqrt{d}} \) is large.

We will find this $S$ by considering all the spheres around the point $\vec 0 $. For \( 0 \leq k \leq d \), let $B_k$ be the open ball in $Q_d$ of radius $k$ around $\vec 0$, and similarly let $S_k$ denote the corresponding sphere of radius $k$. Note that 
\[
|S_k| = {d \choose k} \leq {d \choose \lceil \frac{d}{2} \rceil} < \frac{\sqrt{2} e}{\sqrt{\pi}} \frac{2^d}{\sqrt{d}}.
\]If there exists \( k \) such that \( |S_k| > 0.1n \), then \( n < \frac{10 \sqrt{2} e}{\sqrt{\pi}} \frac{2^d}{\sqrt{d}} \leq C \frac{2^d}{\sqrt{d}} \) and we finish. Therefore, we assume \( |S_k| \leq 0.1n \) for all \( k \).

For a subset \( A \subseteq \mathbb{Z}_2^d \), define \( \rho(A) := |(f|_V)^{-1}(A)| \), which represents the number of vertices that \( f \) maps to \( A \). We obtain the following values:
\begin{align*}
    &\rho(B_0) =\rho(\emptyset)= 0,\\
    &\rho(B_d) \geq 0.9n,\\
    &\rho(B_{k+1}) \leq \rho(B_k) + 0.1n.
\end{align*}

Thus, there exists \( 0 < r < d \) such that \( 0.4n \leq \rho(B_r) \leq 0.5n \). Define \( U := (f|_V)^{-1}(B_r) \), then \( 0.4n \leq |U| \leq 0.5n \), and hence \( |\mathcal{E}(U, U^C)| \geq h |U| \geq 0.4hn \). For each edge \( e \in \mathcal{E}(U, U^C) \), the road \( f_e \) passes through \( S_r \). Since \( G \) is a graph with maximal degree at most 3, \textit{Vizing's Theorem} tells us that there exists a subset of \( \mathcal{E}(U, U^C) \) of size at least \( \frac{1}{4} |\mathcal{E}(U, U^C)| \) in which every pair of edges is non-adjacent. Thus all the associated roads are pairwise disjoint and pass through \( S_r \). Therefore, \( |S_r| \geq \frac{1}{4} |\mathcal{E}(U, U^C)| \geq 0.1h n \).

Thus, we have
\[
n \leq \frac{10}{h} |S_r| < \frac{10 \sqrt{2} e}{\sqrt{\pi} h} \frac{2^d}{\sqrt{d}} \leq C \frac{2^d}{\sqrt{d}}.
\]
\end{proof}

\subsection{Generalized upper bound}
\begin{theorem}[upper bound of theorem \ref{thm: product graph minor-universality}]
Let \( 2\leq k\in \N \), there exists a constant \( C = C(k) \) that satisfies the following: Let \( n \in \mathbb{N} \) and let \( G_1, \ldots, G_n \) be sequence of graphs, each with \(\leq k \) vertices. The graph \( G_1 \Box \ldots \Box G_n \) is not $\frac{Ck^n}{\sqrt n}$-minor-universal.
\end{theorem}

\begin{proof}
We rely on the proof of the previous theorem; the difference in the current proof lies in the need for slightly more computations.

To simplify the computations, we replace \( k \) with \( k+1 \). Thus, we now assume \( k \geq 1 \), each graph \( G_i \) has \(\leq k+1 \) vertices, and we aim to construct a graph with \(\leq \frac{C(k+1)^n}{\sqrt{n}} \) edges that is not a minor of \( G_1 \Box \cdots \Box G_n \).

Consider the graph \( G := \underbrace{K_{k+1} \Box \cdots \Box K_{k+1}}_{\text{n times}} \) where $K_{k+1}$ is the complete graph. Since \( G_1 \Box \cdots \Box G_n \leq G \), finding a graph with \(\leq \frac{C(k+1)^n}{\sqrt{n}} \) edges and with no isolated vertices, that is not a minor of \( G \), ensures that it is also not a minor of \( G_1 \Box \cdots \Box G_n \), which is our goal. This is exactly what we are going to do.

We label the vertices of each \( K_{k+1} \) by \( \{0, \ldots, k\} \), so the vertices of \( G \) correspond to strings in \( \{0, \ldots, k\}^n \). Referring to the proof of theorem \ref{thm: upper bound}, it suffices to bound the sizes of spheres in \( G \). Specifically, we need to show the existence of a constant \( C = C(k) \) such that every sphere in \( G \) centered at \( \vec{0} \) has size \(\leq  \frac{C(k+1)^n}{\sqrt{n}} \).
 
Let \( 0 \leq r \leq n \). A sphere of radius \( r \) around \( \vec{0} \) has size \( \binom{n}{r}k^r \). We need to show that there exists a constant \( C = C(k) \) such that \( \binom{n}{r}k^r \leq \frac{C(k+1)^n}{\sqrt{n}} \) for all \( 0 \leq r \leq n \).

The easy cases \( r \in \{0, n\} \) will be handled separately. For now, fix \( 1 \leq r \leq n-1 \) and let \( a = \frac{r}{n} \). Using the inequality 
\[
\sqrt{2 \pi n} \left( \frac{n}{e} \right)^n e^{\frac{1}{12n+1}} < n! < \sqrt{2 \pi n} \left( \frac{n}{e} \right)^n e^{\frac{1}{12n}},
\]
(from \cite{R}), we obtain
\[
\binom{n}{an}k^{an} \leq \frac{c}{\sqrt{a(1-a)n}} \left( \frac{k^a}{a^a (1-a)^{1-a}} \right)^n,
\]
where \( c = \frac{e}{\sqrt{2\pi}} \). Define \( f(a) := \frac{k^a}{a^a (1-a)^{1-a}} \). We then need to prove the existence of a constant \( C = C(k) \) such that
\[
\frac{c f(a)^n}{\sqrt{a(1-a)n}} \leq \frac{C(k+1)^n}{\sqrt{n}}
\]
for every \( a \in \left\{\frac{1}{n}, \ldots, \frac{n-1}{n}\right\}\).

By differentiating $f$ and comparing it to $0$, one can conclude that: for \( x \in (0,1) \), we have \( f(x) \leq k+1 \), with equality \( f(x) = k+1 \) if and only if \( x = \frac{k}{k+1} \). Moreover, $f'(x)>0$ on $(0,\frac{k}{k+1})$ and $f'(x)<0$ on $(\frac{k}{k+1},1)$.

We proceed by selecting a value \( \frac{k}{k+1} < k' < 1 \) and defining $m := \sup_{a \in (0, 1-k'] \cup [k', 1)} f(a) $.
Note that \( m= \max(f(k'), f(1-k')) < k+1 \), and notice it is just a constant, independent of $n$.

We now derive the desired inequalities: 
\begin{enumerate}
    \item For \( a \in \{0, 1\} \), we have
    \[
    \binom{n}{an} k^{an} \leq k^n \leq C_1 \frac{(k+1)^n}{\sqrt{n}},
    \]
    for some constant \( C_1 = C_1(k) \).

    \item For \( a \in \{\frac{1}{n}, \ldots, \frac{n-1}{n}\} \cap ((0, 1-k'] \cup [k', 1)) \), we have
    \[
    \binom{n}{an} k^{an} \leq \frac{c f(a)^n}{\sqrt{a(1-a)n}} \leq \frac{c m^n}{\sqrt{\frac{n(n-1)}{n^2}}} \leq C_2 \frac{(k+1)^n}{\sqrt{n}},
    \]
    for some constant \( C_2 = C_2(k) \).

    \item For \( a \in \{\frac{1}{n}, \ldots, \frac{n-1}{n}\} \cap [1-k', k'] \), we have
    \[
    \binom{n}{an} k^{an} \leq \frac{c f(a)^n}{\sqrt{a(1-a)n}} \leq \frac{c}{k'(1-k')} \frac{(k+1)^n}{\sqrt{n}}.
    \]
\end{enumerate}

This completes the proof as required.
\end{proof}


\begin{thebibliography}{unsrt}

\bibitem[BCE80]{BCE}
B. Bollobás and P.A. Catlin and P. Erdős,
\textit{Hadwiger's Conjecture is True for Almost Every Graph},
European Journal of Combinatorics, Volume 1, Issue 3, Pages 195-199, 1980.

\bibitem[BDGZ24]{BDGZ}
I. Benjamini, Y. Dikstein, R. Gross, M. Zhukovskii,
\textit{Randomly Twisted Hypercubes: Between Structure and Randomness},
Random Structures \& Algorithms, Volume 66, Issue 1, 2024.

\bibitem[BH21]{BH}
B.~Barrett and D.~Hume.
\newblock Thick embeddings of graphs into symmetric spaces via coarse geometry.
\newblock Preprint available from arXiv:2112.05305.

\bibitem[BST12]{BST}
I.~Benjamini, O.~Schramm, and {\'A}.~Tim{\'a}r.
\newblock On the separation profile of infinite graphs.
\newblock {\em Groups Geom. Dyn.}, 6(4):639--658, 2012.

\bibitem[CS07]{CS}
L. S. Chandran and N. Sivadasan, \textit{On Hadwiger’s conjecture for graph products}, 
Discrete Mathematics, 307(2):266–273, 2007.

\bibitem[EKK23]{EKK}
J. Erde, M. Kang and M. Krivelevich, \textit{Expansion in supercritical random subgraphs of the hypercube and its consequences}, 
Annals of Probability 51 (2023), 127-156.

\bibitem[FKO08]{FKO}
N. Fountoulakis, D. Kühn, and D. and Osthus,
\textit{Minors in random regular graphs},
Random Structures \& Algorithms, Volume 35, Issue 4, Pages 444 - 463, 2008.

\bibitem[Gu16]{Gu}
L. Guth, \emph{Recent progress in quantitative topology}, Surveys in Differential Geometry, \textbf{22}(1):191–216, 2017.

\bibitem[GG12]{GG}
M.~Gromov and L.~Guth.
\newblock Generalizations of the Kolmogorov-Barzdin embedding estimates.
\newblock \emph{Duke Math. J.}, 161(13):2549--2603, 2012.

\bibitem[HLW06]{HLW}
S. Hoory, N. Linial, and A. Wigderson,
\textit{Expander Graphs and Their Applications},
Bulletin of the American Mathematical Society, Volume 43, Number 4, Pages 439–561, October 2006.

\bibitem[Ka24]{Ka}
O.~Kalifa.
\newblock Thick embeddings into the Heisenberg group and coarse wirings into groups with polynomial growth.
\newblock {\em arXiv preprint arXiv:2410.20956v1}, 2024.

\bibitem[Ko24]{Ko}
E. Kowalski, \emph{An Introduction to Expander Graphs}, ETH Zürich, Version of January 8, 2024. Available at \texttt{kowalski@math.ethz.ch}.

\bibitem[Kr19]{Kr}
M. Krivelevich, \textit{Expanders - how to find them, and what to find in them}, 
in Surveys in Combinatorics 2019, A. Lo et al., Eds., London Mathematical Society Lecture Notes 456 (2019), pp. 115-142.


\bibitem[KB93]{KB}
A.~N.~Kolmogorov and Y.~M.~Barzdin.
\newblock On the realization of networks in three-dimensional space.
\newblock In \emph{Selected Works of Kolmogorov}, Kluwer,
Dordrecht, 3:194--202, 1993.


\bibitem[KN21]{KN}
M. Krivelevich and R. Nenadov, \textit{Complete minors in graphs without sparse cuts}, 
International Mathematics Research Notices 2021, No. 12, pp. 8996-9015.


\bibitem[KR96]{KR}
J. Kleinberg and R. Rubinfeld, \textit{Short paths in expander graphs}, 
Proceedings of the 37th Symposium on Foundations of Computer Science (FOCS'96), 1996, pp. 86–95.

\bibitem[KS06]{KS}
M. Krivelevich, B. Sudakov,
\textit{Minors in Expanding Graphs},
Geometric and Functional Analysis, Volume 19, pages 294–331, 2009.

\bibitem[N17]{No}
Sergey Norin,
\textit{Graph Minor Theory},
Lecture notes for the topics course on Graph Minor theory, March 13, 2017, Winter 2017.

\bibitem[R55]{R}
H.~Robbins, ``A remark on Stirling's formula,'' \emph{Mathematical Notes}, edited by F.~A.~Ficken, University of Tennessee, University of Tennessee, Knoxville, Tenn., 1955.


\end{thebibliography}
\end{document}